\documentclass[10pt,a4paper]{article}
\usepackage[latin1]{inputenc}
\usepackage{amsmath}
\usepackage{amsfonts}
\usepackage{amssymb}
\usepackage{amsthm}
\usepackage[english]{babel}
\usepackage{amsmath}
\usepackage{amsfonts}
\usepackage{amssymb}
\usepackage{makeidx}
\usepackage{graphicx}
\newtheorem{lemma}{Lemma}

\newtheorem{theorem}{Theorem}

\newtheorem{rem}{Remark}
\newtheorem{conj}{Conjecture}
\newtheorem{defi}{Definition}
\title{Kolakoski Sequence: Links between Recurrence, Symmetry and Limit Density}
\author{Alessandro Della Corte\footnote{Mathematics Division, School of Sciences and Technology, University of Camerino (Italy). Email:alessandro.dellacorte@unicam.it}}
\begin{document}
\setlength{\parindent}{0pt}
\maketitle
\begin{abstract}
The Kolakoski sequence $S$ is the unique element of $\left\lbrace 1,2 \right\rbrace^{\omega}$ starting with 1 and coinciding with its own run length encoding. We use the parity of the lengths of particular subclasses of initial words of $S$ as a unifying tool to address the links between the main open questions - recurrence, mirror/reversal invariance and asymptotic density of digits.  In particular we prove that recurrence implies reversal invariance, and give sufficient conditions which would imply that the density of 1s is $\frac{1}{2}$. 
 


\end{abstract}
\section{Introduction}
In 1939 Oldenburger considered, within the context of symbolic dynamics, a sequence having the property of coinciding with its own run length encoding  \cite{oldenburger1939exponent}. If we choose  the alphabet $\left\lbrace 1,2 \right\rbrace$ there are two such sequences, the second of which being simply the first one without the initial element. The two sequences start as

$$1221121221\ldots$$ and $$2211212212\ldots$$
In 1965 Kolakoski rediscovered the sequence \cite{kolakoski1965problem}, and it was easily established that it is not eventually periodic. Besides this, very little is known about the sequence. In particular, it is still not known whether it is recurrent, whether it has basic symmetry properties (mirror/reversal invariance) and whether the asymptotic density of 1s exists and equals $\frac{1}{2}$, a conjecture formulated by Keane \cite{keane1991ergodic}. A sharp bound ($0.5 \pm 0.00084$) for the density of 1s has been provided  \cite{chvatal1994notes}. Concerning other properties of the sequence, it has been proved that it is cube-free, which is a particular case of a more general result on repetitions \cite{carpi1994repeated}. Moreover, a measure conjectured to completely describe the densities of all subwords of the sequence has been introduced, and the conjecture has been proved under fairly natural additional hypotheses \cite{dekking1997long}. Recursive formulas for the $n$-th element of the sequence are also known \cite{steinsky2006recursive,hammam2016some}.

The sequence is relevant for applications concerning optical properties of aperiodic structures \cite{tuz2013optical,fesenko2014terahertz,sing2004kolakoski}, but probably its most interesting features are linked to the unique combination of the simplicity of its definition and the difficulty of the problems it raises. 

The sequence is nowadays indexed as A000002 in Sloane's online \textit{Encyclopedia of Integer Sequences}.

In this paper we study the open problems - namely recurrence, mirror and reversal invariance and asymptotic frequency of digits - trying to identify a unifying concept, i.e. the parity of the integrals (the converse transform of run length encoding) of subwords of the sequence. After introducing some notation and terminology in Section \ref{not}, the main open problems are reformulated in terms of parity of integrals of prefixes in Section \ref{refo}, while in Section \ref{mr} we prove some results, including that recurrence implies reversal invariance, and provide sufficient conditions implying that the asymptotic frequency of 1s is $\frac{1}{2}$. In Section \ref{ite} we describe a constructive procedure showing the existence of arbitrarily long recurrent subwords and identify places where they must occur in the structure of $S$. Finally, in Section \ref{more} we formulate some conjectures arising from the proposed approach. 

We tried to make the paper as self-contained as possible. For this reason we included the proof of some facts (covered in Lemmas \ref{lemma1}, \ref{furtherintegration}, \ref{findderivativeS}, \ref{find_derivative}, \ref{Cinfty} and \ref{lengthestimates})  already known, although generally not presented, in the cited literature, exactly in the form proposed herein. In this way we could also ``optimize" their exact formulation for our aims.

\section{Notation and preliminary definitions}
\label{not}
Let $\mathcal{A}^*$ be the set of the finite words on the alphabet $$\mathcal{A}=\left\lbrace 1,2 \right\rbrace$$ and $\mathcal{A}^{\infty}$ the set $\mathcal{A}^*\cup \mathcal{A}^{\omega}$ of all finite or infinite words on $\mathcal{A}$. We let $\epsilon$ denote the empty word and we set $\mathcal{A}^+:=\mathcal{A}^*\setminus \left\lbrace\epsilon\right\rbrace$. 

The concatenation of the finite word $w=a_1a_2\cdots a_n$ and the (possibly infinite) word $v=b_1b_2\cdots$, i.e. the word $a_1a_2\cdots a_nb_1b_2\cdots$, is written as $wv$. The sets $\mathcal{A}^*$ and $\mathcal{A}^+$ have respectively the structure of a free monoid and a free semigroup with the internal operation defined as the concatenation of words.

For every $w\in\mathcal{A}^{+}$, by $\widetilde{w}$ we mean the \textit{mirror word }of $w$, i.e. the transform of $w$ under the substitutions $$1\rightarrow 2$$ $$2\rightarrow 1$$ 
We also set $\widetilde{\epsilon}:=\epsilon$. If $w=a_1\cdots a_n$ is a word in $\mathcal{A}^+$, we set $\Sigma w:= \Sigma_{i=1}^n a_i$ and $\overleftarrow{w}:=a_na_{n-1}\cdots a_1$, calling the former the \textit{sum} of $w$ and the latter the \textit{inverse word }of $w$ (we set $\overleftarrow{\epsilon}:=\epsilon$). We indicate by $|w|$ the length of $w$, i.e. the positive integer $n$ (we set $|\epsilon|:=0$). 

We say that $v\in\mathcal{A}^+$ is a \textit{subword} of $w=a_1a_2\cdots\in \mathcal{A}^{\infty} $ if there exist a positive integer $k$ and a non-negative integer $h$ such that $v=a_k a_{k+1}\cdots a_{k+h}$. In the following it will be handy to have a term concisely referring to a particular occurrence of a subword. Therefore, we call the pair $(v,k)$  a \textit{subrow} of $w$ and say that $(v,k)$ is an \textit{occurrence} of $v$ in $w$, and that two subrows $(v_1,k)$ and $(v_2,h)$ \textit{coincide as subwords} if $v_1=v_2$. When we want to emphasize the initial and final elements of the subrow $a_ka_{k+1}\cdots a_{k+h}$, we write it as $w_{k,k+h}$. 
To lighten the notation, if there is no possibility of confusion we may use the same symbol for the subrow $(v,k)$ and the subword $v$. We let $\mathcal{SR}(S)$ denote the set of all the finite subrows of $S$.

We say that $v$ is a \textit{prefix} of a (possibly infinite) word $w=a_1a_2\cdots$ if there is a positive integer $k$ such that $v=a_1\cdots a_k$. We say that $v$ is a \textit{suffix} of a finite word $w=a_1\dots a_n$ if there is a positive integer $k< n$ such that $v=a_{n-k}a_{n-k+1}\dots a_n$.

\vspace{0.2cm}

Let us define a map from $\mathcal{A}^{\infty}$ to itself by means of alternating substitution rules. Specifically, for every nonempty $w\in \mathcal{A}^{\infty}$ we define the following substitution rules:
\begin{align}
\begin{split}
&1\rightarrow 1\qquad\quad
2\rightarrow 11\qquad\quad
\text{for the elements having odd index in}\quad w\\
&1\rightarrow 2\qquad\quad
2\rightarrow 22\qquad\quad
\text{for the elements having even index in}\quad w
\end{split}
\label{subs}
\end{align}
We let $w^{-1}$ denote the transform of $w$ under the substitutions \eqref{subs} \footnote{The reason of this choice is that the converse transformation, which we introduce later, is usually denoted derivative-like, with positive integers as exponents. Moreover it will come handy, compared to something more cumbersome like $I(\cdot)$, when we will have to write relatively long concatenations of words which are iterations of the integration map.}. We also set $\epsilon^{-1}:=\epsilon$. For every $w\in\mathcal{A}^{\infty}$, we define inductively: 
\begin{align}
\begin{split}
&w^{0}:=w\\ 
&\text{for every integer}\, k > 1, \quad w^{-k}:=(w^{-(k-1)})^{-1} 
\end{split}
\label{integral}
\end{align}
We refer to the $(\cdot)^{-1}$ map as the \textit{integration} map.

Alternating substitution rules are quite well investigated and have also been generalized \cite{culik1992iterative}. Relative results have been used specifically to study the Kolakoski sequence. It was indeed proved that, even if \eqref{subs} are very close to the simplest possible case of alternating substitution rules, its fixed point (i.e., Kolakoski sequence) cannot be obtained by iteration of a \textit{simple} substitution \cite{culik1992alternating}.  
%
%\begin{rem} Notice that in general two subrows of a given word, even if they coincide as subwords, can have different integrals (that is, $(\cdot)^{-1}$ is a non-morphic operator). For instance, if $|u|$ is odd, $(uv)^{-1}=u^{-1}\widetilde{v^{-1}} \ne u^{-1}v^{-1}$. 
%\label{remsubs}
%\end{rem} 

The existence and uniqueness of the Kolakoski sequence are established by means of the following Lemma. 
\begin{lemma}
There exists a unique element $S$ of $\left\lbrace 1,2 \right\rbrace^{\omega}$ such that $S=S^{-1}$. Moreover, indicating the $n$-th element of $S$ by $s_n$, for every positive integer $n$ there exists a positive integer $h$ such that $s_n$ is the $n$-th element of $(12)^{-k}$ for every integer $k$ such that $k\ge h$.
\label{lemma1}
\end{lemma}
\begin{proof}
Take a finite word $u$ such that $$u^{-1}=uv_1$$ with $v_1\ne\epsilon$. Integrating both sides of the previous equality one gets $$u^{-2}=(uv_1)^{-1}=u^{-1}v_2=uv_1v_2$$ where $v_2$ equals $v_1^{-1}$ or $\widetilde{v_1^{-1}}$ according to $|u|$ being respectively even or odd. Iterating the argument it follows that, for every non-negative integer $k$,  
\begin{equation}
u^{-k}=uv_1\cdots v_k
\label{blocks}
\end{equation}
where, for $2\le h\le k$, $v_h=v_{h-1}^{-1}$ or $v_h=\widetilde{v_{h-1}^{-1}}$ according to $|uv_1\cdots v_{h-2}|$ being respectively even or odd.
Since the words $v_i$ are nonempty, an arbitrarily long prefix of $u^{-k}$ remains unaltered by further integrations. More precisely, if we write $u^{-k}$ as $$u^{-k}=a_1^ka_2^k\cdots a_{|u^{-k}|}^k$$ (where $a_i ^k\in \left\lbrace 1,2\right\rbrace$), for every positive integer $n$ there is $h$ such that $a_n^{j_1}=a_n^{j_2}$ for every $j_1,j_2\ge h$. Hence we can define the limit sequence $S$ of the right hand side of \eqref{blocks} for $k\rightarrow\infty$, which clearly verifies $S=S^{-1}$. Taking $u=12$ one gets the existence of $S$. As for the uniqueness, it follows immediately observing that the only word of length 2 which is a prefix of its integral is 12.     
\end{proof}
\begin{defi}
The sequence $S$ is called the \textit{Kolakoski sequence}. 
\end{defi}
By the arbitrariness of the prefix $u$ in the previous proof, we easily get by induction the following
\begin{lemma}
If $p$ is a prefix of $S$, such is $p^{-k}$ for every non-negative integer $k$. 
\label{furtherintegration}
\end{lemma}
We now want to adapt the previous definition of integral so that it applies nicely to \textit{subrows} of $S$, meaning that we can identify which subrow of $S$ can be naturally seen as the integral of a given subrow. We thus want a version of the integration map which maps $\mathcal{SR}(S)$  to itself (while $(\cdot)^{-1}$ maps $\mathcal{A}^{\infty}$ to itself).  Therefore we introduce the following
\begin{defi}
Let $w$ be a subrow of $S$ and $u$ the prefix such that $S=uw\dots$. We define the $S$-integral of the subrow $w$ as the subrow $w^{-1}_{S}:=S_{h,k}$ where $$h=|u^{-1}|+1\quad\text{and} \quad k=|u^{-1}|+|w^{-1}|$$
We also define inductively $w^{-n}_S:=\left(w^{-n+1}_S\right)^{-1}_S$.
\label{subsint}
\end{defi}
\begin{rem} Since $(\cdot)^{-1}$ is a non-morphic map, the $S$-integral of a subrow $w$ does not coincide always with its integral as a subword, defined by means of the substitution rules \eqref{subs}. Indeed, if $uw$ is a prefix of $S$, considering $w^{-1}_S$ as a subword, we have
$$w^{-1}_{S}=w^{-1}$$ if $|u|$ is even and $$w^{-1}_{S}=\widetilde{w^{-1}}$$ if $|u|$ is odd. Notice also that in general, for a subrow $w$ which is not a prefix and for $k$ large, $w^{-k}$ and $w_S^{-k}$ are different words which are not linked in any trivial way.
\label{remsint}
\end{rem}

Next we want to define the property of a subrow of having $S$-integrals of even length up to a certain order, starting from the order 0 (that is, from the length of the subrow itself).

More precisely, we introduce the following
\begin{defi}
We say that the subrow $w$ 
%$S_{(|u|+1),|w|}$ 
is $k$\textit{-regular} if 
%$|(uw)^{-h}|-|u^{-h}|$ 
$|w^{-h}_S|$ is even for $0\le h\le k$. We say that a subrow is $k$\textit{-normal} if it is $k$-regular but not $(k+1)$-regular. We say that a subrow is $\infty$\textit{-regular} if it is $k$-regular for every non-negative integer $k$.
\end{defi}
Notice that in case $w$ is a prefix, $k$-regularity reduces to requiring that $|w^{-h}|$ is even for $0\le h\le k$.

We indicate by $k$-R the subset of $\mathcal{SR}(S)$ consisting of all the $k$-regular subrows of $S$, by $k$-N the subset of $\mathcal{SR}(S)$ consisting of all the $k$-normal subrows of $S$ and by $\infty$-R the subset of $\mathcal{SR}(S)$ consisting of all the $\infty$-regular subrows of $S$.
It is easily proved the following
\begin{lemma}
\begin{enumerate}
\item $k\ge h \implies k$-\emph{R} $\subseteq h$-\emph{R}.
\item For every non-negative integer $k$, $k$-\emph{-N}$\subset k$-\emph{-R}.
\item $k\ne h \implies k$-\emph{N} $\cap\,\, h$-\emph{N}$=\emptyset$.
\item For every non-negative integer $k$, $w\in k$-\emph{N}$ \implies w \notin \infty$-\emph{R}.
\item For positive integers $a<b<c$, $S_{a,b}, S_{(b+1),c} \in k$-\emph{R} $\implies S_{a,c}\in k$-\emph{R}.
\item For positive integers $a<b<c$, $S_{a,b}, S_{(b+1),c} \in k$-\emph{N} $\implies S_{a,c}\in (k+1)$-\emph{R}.
\end{enumerate} 
\label{Syntax}
\end{lemma}
Next we want to introduce the converse operation of integration, i.e. the so called \textit{derivative} for words in $\mathcal{A}^*$, which, roughly speaking, coincides with a run-length counting operation. However, we should take care to avoid the ambiguity arising when a subword starts or ends with a single digit not belonging to a pair of equal elements of the alphabet, as in that case we cannot know the length of its run without looking outside the subword. For this reason it is usual \cite{dekking1997long} to cut off those single digits, when they are present. 

More precisely, for every $w=a_1\cdots a_n\in\mathcal{A}^+$ we define $w'$ as the unique finite word such that $(w')^{-1}$ equals
\begin{align*}
& a_1\cdots a_n\quad \text{if}\quad a_1= a_2\quad \text{and} \quad a_{n-1}= a_n\\
& a_2\cdots a_n \quad \text{if}\quad a_1\ne a_2\quad \text{and}\quad a_{n-1}= a_n\\
& a_1\cdots a_{n-1} \quad \text{if}\quad a_1= a_2\quad \text{and}\quad a_{n-1}\ne a_n\\
& a_2\cdots a_{n-1} \quad \text{if}\quad a_1= a_2\quad \text{and}\quad a_{n-1}= a_n
\end{align*}  
We also set $1'=2'=\epsilon':=\epsilon$, so that derivatives of every order exist for all elements of $\mathcal{A}^*$; notice that this also implies $(12)'=(21)'=\epsilon$. We also define inductively $w^{(n)}:=(w^{(n-1)})'$. Finally, we define the derivative of an infinite word $v=a_1a_2\cdots \in\mathcal{A}^{\omega}$, not ultimately constant, as the word whose $m$-th element is the $m$-th element of $(a_1\cdots a_{n})'$ for all sufficiently large $n$. It is easily seen that this $m$-th element stays the same when $n$ diverges unless the sequence is ultimately constant (in which case the derivative of the sequence is not defined). 

Adopting the usual convention, we indicate by $C^k$ the set of words which belong to $\mathcal{A}^{\infty}$ together with their first $k$ derivatives, and by $C^{\infty}$ the set $\bigcap_{k\in\mathbb{N}} C^k$. 

Similarly to what done before for integrals, we want now to adapt the definition of derivative so that it works for subrows. Therefore we introduce the following

\begin{defi}
Let $w$ be a subrow of $S$. If there exists a subrow $v$ such that $v_S^{-1}=w$, we set $w_S':=v$. We call $v$ the $S$-derivative of $w$. 

We also define inductively $w_S^{(n)}:=(w_S^{(n-1)})_S'$, of course if $w_s^{(k)}$ admits an $S$-derivative for every $k\le n$.
\end{defi}

\begin{rem}
Notice that not every subrow has an $S$-derivative. For instance there is no subrow $u$ such that $s_{3}s_4=u_S^{-1}$. 
\label{Sderiv}
\end{rem}
The following Lemma characterizes the subrows admitting an $S$-derivative.

\begin{lemma}
Let $n$ and $m$ be two positive integers such that $n<m$. If $w=s_n\dots s_{m}$ is a subrow of $S$, it admits an $S$-derivative if and only if  $s_{n-1}\ne s_n$ and $s_{m}\ne s_{m+1}$. Moreover, if $w_S'=s_h\dots s_{h+k}$ then $h\le n$ and $h+k\le m$.
\label{findderivativeS}
\end{lemma}
\begin{proof}
Suppose that $w$ admits an $S$-derivative. Then $w$ is the transform under substitutions \eqref{subs} (or the mirror of the transform under \eqref{subs}) of another subrow $v=s_h \dots s_j$ ($j\ge h$). This means in particular that $(s_j)_S^{-1}=s_m$ or $(s_j)_S^{-1}=s_{m-1}s_{m}$, so that $(s_{j+1})_S^{-1}=s_{m+1}$ or $(s_{j+1})_S^{-1}=s_{m+1}s_{m+2}$. Since $j$ and $j+1$ cannot be both even or both odd, it follows that $s_m \ne s_{m+1}$. The proof proceeds analogously for $s_n$ if $s_{n-1}$ exists, otherwise (i.e. if $w$ is a prefix), the thesis is vacuously true.

Conversely, suppose that $s_{n-1}\ne s_n$ and $s_{m}\ne s_{m+1}$. Then, by definition of $S$, $w$ is the transform under \eqref{subs} (or the mirror of the transform under \eqref{subs}) of some subrow $v$.

Finally, if $w_S'=s_h\dots s_k$, the inequalities  $h\le n$ and $h+k\le m$ follow from the fact that, for every word $w$, $|w^{-1}|\ge |w|$. 
\end{proof}
 Since $w'$ is always a subword of $w_S'$, the following Lemma follows from the previous one:
\begin{lemma}
If $w=s_ns_{n+1}\cdots s_{n+m}$ is a subrow of $S$ with nonempty derivative, there exists $h$ and $k$ such that $w'=s_{h}s_{h+1}\cdots s_{h+k}$.
\label{find_derivative}
\end{lemma}

 The previous Lemma immediately implies that
\begin{lemma}
If $w$ is a subword of $S$, then $w\in C^{\infty}$.
\label{Cinfty}
\end{lemma}

Finally we define formally what is meant by asymptotic frequency of digits. 
Indicating by $|v|_x$ the number of occurrences of the digit $x$ ($x\in\left\lbrace 1,2 \right\rbrace$) in $v\in \mathcal{A^+}$, and by $f_v(x)$ the frequency of $x$ in $v$, i.e. the number $\frac{|v|_x}{|v|}$, we set $$f_{\infty}(x):=\lim_{n\rightarrow \infty} \frac{|S_{1,n}|_{x}}{n}$$ whether the limit exists. The most famous conjecture concerning Kolakoski sequence is Keane's conjecture  \cite{keane1991ergodic}: $$f_{\infty}(1)\quad \text{exists and equals}\quad \frac{1}{2}$$

\section{Reformulation of the problems}
\label{refo}

In this section we reformulate some open questions concerning $S$ in terms of regularity/normality of subrows. Let us recall that, according to \eqref{integral}, elements with even (odd) index in $S$ are mapped by the integration in 2 or 22 (1 or 11). We use systematically this fact (usually without mentioning it explicitly) throughout. 

 In the following we will need a (rough) estimate of the relative length of $w$, $w'$ and $w^{-1}$ when $w$ is a subword of $S$, ensuring in particular that, for every finite word $w$ with more than one element, 
\begin{equation}
|w^{(k)}|\rightarrow 0\quad \text{ if $k$ diverges}
\label{derivativelength}
\end{equation}
and
\begin{equation}
|w^{-k}|=|w_S^{-k}|\rightarrow \infty\quad \text{ if $k$ diverges}
\label{integrallength}
\end{equation}
and that for every positive integer $k$,
\begin{equation}
|w^{(k)}|\rightarrow \infty\quad \text{if $|w|$ diverges}
\label{derivativediverges}
\end{equation}
This is obtained observing that $|w^{-1}|=\sum w$, and that the maximum and minimum density of 2s in a $C^{\infty}$ word are achieved respectively by $11211$ and $22122$. From this (recalling that the derivative cuts off single digits at both ends) the following Lemma is easily proved. 
\begin{lemma}
If $w$ is a subrow of $S$ and $|w|\ge 3$, then $$\frac{6}{5}|w| \le |w^{-1}|=|w_S^{-1}|\le \frac{9}{5}|w|,$$ $$\frac{5}{9}|w| \le |w_S'|\le \frac{5}{6}|w|\quad(\text{if} \quad w \quad \text{admits}\,\, \text{an}\,\, S\text{-derivative})$$ and $$\frac{1}{4}|w|\le |w'|\le \frac{5}{6}|w|$$  
%$|w^{(n)}|\rightarrow 0$ when $n\rightarrow\infty$.
%$|w_n^{(k)}|\rightarrow \infty$ when $|w_n|\rightarrow \infty$.
\label{lengthestimates}
\end{lemma}
 The asymptotic behaviors \eqref{derivativelength},  \eqref{integrallength} and \eqref{derivativediverges} immediately follow from Lemma \ref{lengthestimates} (notice that, if $|w|=2$, $|(w_S)^{-2}|\ge 3$). 

We add some definitions: a sequence $w\in \mathcal{A}^{\omega}$ is called \textit{recurrent} if every finite subword of it is repeated (and therefore every finite subword is repeated infinitely many times). It is called \textit{uniformly recurrent} if it is recurrent and the gaps between consecutive occurrences of every given finite subword are bounded. Moreover, $w$ is called \textit{mirror invariant} (\textit{reversal invariant}) if the set of its finite subwords is closed under the mirror operation: $v\rightarrow \widetilde{v}$ (inverse operation: $v\rightarrow \overleftarrow{v}$). 

It is a well known result that, for $S$, mirror invariance implies recurrence \cite{dekking1997long}. The converse implication is not trivial, and sufficient conditions for it to hold are provided in Theorem \ref{thm1}. For this, though, we need some preliminary results.

The links between recurrence, mirror invariance and regularity/normality of subrows  are established in the following Lemmas.

\begin{lemma}
$S$ is recurrent if and only if for every positive integer $k$ there is a $k$-regular prefix of $S$.
\label{recur}
\end{lemma}
\begin{proof}
Suppose that we can find a $k$-regular prefix $w$ of $S$ for every non-negative integer $k$.

First of all notice that $w\in k$-R implies that $|w|>2$, which in turn implies that $|w^{-a}|$ is strictly larger than $|w^{-b}|$ for every choice $a,b$ of positive integers such that $a>b$ (as there are no runs of consecutive 1s longer than 2 in $w$). 
From $w\in k$-R it follows that (as soon as $k \ge 1$) $|w|$ and $|w^{-1}|$ are even prefixes of $S$, so that $s_{|w|+1}$ and $s_{|w^{-1}|+1}$ are both odd-indexed elements of $S$. Then from Lemma \ref{furtherintegration}, recalling the substitution rules \eqref{subs}, it follows that  $w^{-1}1$ and $w^{-2}12$ are both prefixes for $S$. 
 
Integrating further we find a prefix of the form
\begin{equation}
(w^{-2}12)^{-k+2}=w^{-k}(12)_S^{-k+2}=w^{-k}(12)^{-k+2}
\label{recurrence}
\end{equation}
where the last equality is again due to the fact that $w\in k$-R.
By Lemma \ref{blocks} the last factor in the right hand side of \eqref{recurrence} coincides with a prefix of $S$. Recalling \eqref{integrallength}, this prefix is arbitrarily long  if $k$ is large enough, which is sufficient to conclude that $S$ is recurrent. 

Conversely, suppose that $S$ is recurrent. This implies, in particular, that arbitrarily long prefixes of $S$ are repeated infinitely many times, thus for every positive integer $N$ there is a prefix $w$ and a prefix of the form $wvw$  such that both $|w|$ and $|v|$ are larger than $N$. By suitably selecting $w$, we can assume that the last element of $w$ is not equal to the first element of $v$. Moreover, since $w$ starts with 12211, by Lemma \ref{Cinfty} the last element of $v$ has to be different from the first element of $w$, as otherwise $vw\notin C^{\infty}$. Therefore, by Lemma \ref{findderivativeS}, there exists a nonempty word $u_1$ such that, setting $p_1:=(w_S)'$, the word $p_1 u_1 p_1$ is also a prefix for $S$. 

 We then define recursively $$p_{i+1}:=(S_{1,|p_i|})_S'$$ and $$u_{i+1}:=(u_i)_S'$$ in case $s_{|p_i|}\ne s_{|p_i|+1}$. If instead $s_{|p_i|}= s_{|p_i|+1}$ we replace, in the definition of $p_{i+1}$ and $u_{i+1}$,  $p_i$ with the largest of its prefixes admitting an $S$-derivative and $u_{i}$ with the smallest subrow having $u_{i}$ as a suffix and admitting an $S$-derivative. 

 Since $|v|$ can be arbitrarily large and recalling \eqref{derivativediverges}, we have that $$p_k u_k p_k$$ is a prefix of $S$ (with $u_k\ne \epsilon$) for every $k$ such that $|p_k|\ge 2$. Therefore $p_k$ starts with 1 for every $k$ such that $|p_k|\ge 2$. Since $p_k$ also follows the prefix $p_ku_k$ in $S$, this means, recalling the substitution rules \eqref{subs}, that the first $k-1$ integrals of the prefix $p_{k} u_{k}$ have even length. By Lemma \ref{lengthestimates} it follows that $k \rightarrow \infty$ when $|w|\rightarrow \infty$.   
\end{proof}
\begin{lemma}
$S$ is mirror invariant if and only if for every positive integer $n$ there is a $k$-normal prefix of $S$ with $k>n$.
\label{mirroring}
\end{lemma}
\begin{proof}
 Suppose that $w$ is a $k$-normal prefix of $S$. As seen in the proof of Lemma \ref{recur}, $w^{-h}(12)^{-h+2}$ is also a prefix for $S$ for every $h\le k+1$. Since $|w^{-k-1}|$ is odd, integrating further and recalling Lemma \ref{furtherintegration} one gets the prefix $w^{-k-2}\widetilde{v}$, where $v=(12)^{-k}$.  By Lemma \ref{lemma1}, $v$ is also a prefix of $S$, and $|v|$ is arbitrarily large if $k$ is large enough. Since concatenation commutes with the mirror operation (that is: $\widetilde{uv}=\widetilde{u}\widetilde{v}$ for every $u,v \in \mathcal{A^+}$), this is sufficient to conclude that $S$ is mirror invariant. 

 Conversely, suppose that $S$ is mirror invariant and let $w$ be a prefix of $S$ such that its last element is not equal to the following element of $S$. By mirror invariance there is a prefix of the form $wv\widetilde{w}$. Let us define the subwords $p_n$ and $u_n$ as done in the previous proof, and let $\bar{n}$ be the largest integer for which $|p_n|\ge 2$. Since $S$-derivatives of mirror words coincide as subwords, there are prefixes of the form $$p_n u_n p_n$$ with $u_n$ nonempty for every positive integer $n\le\bar{n}$ (notice that this means that $S$ is recurrent). As $p_n$ is a prefix of $w$ and so starts with 1 for every $n\le\bar{n}$, \eqref{derivativediverges} ensures that the prefix $p_n u_n$ is $k$-regular for arbitrarily large positive integers $k$ if $|w|$ and $|v|$ are chosen large enough.  Moreover, since $\widetilde{w}$ starts with 2 and $(p_1u_1p_1)^{-1}=wv\widetilde{w}$  by hypothesis, $|p_1u_1|$ has to be odd, and therefore the prefix $p_{\bar{n}}u_{\bar{n}}$ is $(\bar{n}-2)$-normal, where $\bar{n}-2$ can be arbitrarily large if $|w|$ is chosen large enough.    
\end{proof}
 It is known that mirror invariance is equivalent to: every $C^{\infty}$ finite word occurs in $S$ \cite{dekking1997long}. One implication is obvious, while the other (whose proof was left to the reader in \cite{dekking1997long}) easily follows from the fact that mirror invariance implies that, if $w$ does not occur in $S$, neither does $w'$. This result and Lemma \ref{mirroring} mean that
\begin{lemma}
Every $C^{\infty}$ word is a subword of $S$ if and only if for every positive integer $n$ there is a $k$-normal prefix of $S$ with $k>n$.
\label{subwords}
\end{lemma}

 Concerning uniform recurrence, we have the following
\begin{lemma}
$S$ is uniformly recurrent if and only if, for every positive integer $n$, $S$ can be written as an infinite concatenation of $k$-regular subrows of bounded length with $k>n$.
\label{uniformly}
\end{lemma}
\begin{proof}
Suppose that, for every positive integer $N$, there exists a sequence of subwords $w_i\, (i\in \mathbb{N})$ and a positive integer $M$ such that $|w_i|<M$ for every $i$ and 
\begin{equation}
S=w_1w_2\dots 
\label{w_i}
\end{equation}
with every $w_i\in k$-R and $k>N$. Then integrating \eqref{w_i} $k$ times yields
\begin{equation*}
S=w_1^{-k}w_2^{-k}\dots 
\end{equation*} 
Since $|w_i^{-h}|$ is even for every $h\le k$, the word $(12)^{-k+2}$ is a prefix of $w_i^{-k}$ for every $i\ge 2$, and by Lemma \ref{lemma1} it is also a prefix of $S$. Since, by Lemma \ref{lengthestimates}, $|w_i^{-k}|<M\left(\frac{9}{5}\right)^k$, it follows that $S$ is uniformly recurrent.

Conversely, suppose that $S$ is uniformly recurrent. Then, for every prefix $w$, $S$ can be written as
\begin{equation}
S=wu^1wu^2w\dots 
\label{unif}
\end{equation} 
with $2<|u^i|<M$ for every $i$ for some $M>0$. We can assume that $w$ ends with $11$ or $22$, so that its last element is not equal to the first element of $u^i$ for every $i$. Since $w$ starts with $12211$, every $u^i$ has to end with $2$, otherwise a subword which is not $C^{\infty}$ would occur in $S$, which is not possible by Lemma \ref{Cinfty}. By Lemma \ref{findderivativeS} we then have
\begin{equation*}
S=w_S'(u_1)_S'w_S'(u_2)_S'w_S'\dots 
\end{equation*}
Taking $w$ is long enough and defining the subwords $p_i$ ($i=1,\dots,k$) as in the proof of Lemmas \ref{recur} and \ref{mirroring} and the subrows $u_{i}^j$ ($j=1,2,\dots$) accordingly, we can iterate $k$ times the argument, so as to obtain
\begin{equation*}
S=p_ku_k^1p_ku_k^2 p_k\dots 
\end{equation*}
Since $M$ can be chosen arbitrarily large (by simply neglecting a suitable number of occurrences of $w$ in $S$ if needed), the limit behaviour \eqref{derivativediverges} means that $|u_k^j|$ can be made nonempty for every positive integer $j$ and for arbitrarily large $k$. Therefore, for every positive integer $n$ and for arbitrarily large $k$ we can find prefixes of the form $$p_ku_k^1\dots p_ku_k^n$$ and noticing that $p_i$ begins with 1 for every positive integer $i<k$, it follows that these prefixes are $(k-1)$-regular. Then, by Lemma \ref{Syntax}, so is every subrow $S_{a,a+b}$ where $a=|p_{k}u_{k}^h|+1$ and $b=|p_{k}u_{k}^{h+1}|$ ($ h=1,2,\dots, k-1$). Recalling the definition of $p_i$ and $u_i^j$, and that differentiation cannot increase the length of subrows, the inequalities $|p_{k}|\le |w|$ and $|u_{k}^j|\le |w|+M$ follow, so that each subrow of type $p_{k}u_{k}^j$ has length not larger than $2|w|+M$, which concludes the proof.  
\end{proof}

\begin{rem}
Lemmas \ref{recur}, \ref{mirroring}, \ref{subwords} and \ref{uniformly} can be straightforwardly adapted to generalized Kolakoski sequences defined over binary alphabets $\left\lbrace m,n\right\rbrace$ other than $\mathcal{A}$ if we define analogously the concepts of regularity and normality of subrows.  
\end{rem}
On generalized Kolakoski words we mention the works by Sing \cite{sing2004kolakoski,sing2010more} and, in an interesting but slightly different direction compared to typical Kolakoski literature, by Shen \cite{shen2018kolakoski}. 
\section{Main results}
\label{mr}
 Let us start by observing that the existence of an $\infty$-regular prefix of $S$ would have strong consequences on its structure and properties, as $S$ would then be recurrent and would have a rigidly fractal structure. 

 More precisely, we establish the following
\begin{theorem}
Suppose that $S$ has an $\infty$-regular prefix $w$ and let $k$ be a positive integer large enough so that $|(12)^{-k-2}|>|w|$. Then, for every positive integer $n$, $S$ has a prefix with the following structure:
\begin{equation}
w^{-nk}w^{-(n-1)k}\cdots w^{-k}w
\label{inftyreg}
\end{equation} 
In particular, $S$ is recurrent.  
\label{infty}
\end{theorem}
\begin{proof}
Since the prefix $w$ is $\infty$-regular, there exists a positive integer $\bar{h}$ such that $$w^{-h}(12)^{-h-2}$$ is also a prefix for every $h \ge \bar{h}$. Therefore arbitrarily long prefixes of $S$ are repeated, which is enough to have recurrence. In particular, if $k$ is such that $|(12)^{-k-2}|>|w|$, then, by Lemma \ref{lemma1}, $$w^{-k}w$$ is a prefix. Integrating further for $k$ times, and recalling that $w\in \infty$-R, we get the prefix $$w^{-2k}w^{-k}w$$ and continuing the integrations for further $(n-2)k$ times we get \eqref{inftyreg}. Notice that, since $w^{-nk}$ is a prefix for every $n$,  it has to begin with $w^{-hk}$ for every $h<k$.  
\end{proof}
\begin{rem}
Theorem \ref{infty} can be applied to generalized Kolakoski words. Its application to Kolakoski words over binary alphabets $\lbrace m,n \rbrace$ in which $m$ and $n$ are both even or both odd (and therefore every prefix of even length if $\infty$-regular) immediately implies that those sequences are recurrent, which is a well-known result already obtained by other means \cite{sing2010more}.
\end{rem}

As already said, it is a known result that if $S$ is mirror invariant, then it is recurrent \cite{dekking1997long} (notice that Lemmas \ref{recur} and \ref{mirroring} immediately imply that). The converse implication is obtained with an additional hypothesis in the following 
\begin{theorem}
 $\infty$-\emph{R}$\,=\emptyset\implies$ ($S$ is recurrent$\implies S$ is mirror invariant).
 \label{thm1}
 \end{theorem}
\begin{proof}
Suppose that  $\infty$-R$\,=\emptyset$ and that $S$ is recurrent. Then by Lemma \ref{recur} there is a strictly increasing sequence of positive integers $k_n$ such that $S$ has a $k_n$-regular prefix $w_n$ for every $n$. Since $w_n\notin \infty$-R, there is a positive integer $k$ which is the least integer such that $|w_n^{-k}|$ is odd. As seen in the proof of Lemma \ref{recur}, $w_n^{-k}(12)^{-k+2}$ is also a prefix of $S$, and integrating once more (recalling Lemma \ref{furtherintegration}) we get $$S=w_n^{-k-1}\widetilde{(12)^{-k+1}}\dots$$ Recalling that $(12)^{-k+1}$ is also a prefix of $S$ by Lemma \ref{lemma1}, and that $k$ can be taken arbitrarily large by suitably choosing $w_n$, we can conclude.
\end{proof}

The implication from reversal invariance to recurrence is a known result holding for all elements of $\mathcal{A}^{\omega}$ (for the application to differentiable sequences see \cite{brlek2003note}, where a stronger result is proved, namely that recurrence of $S$ is implied by the existence of arbitrarily long palindromes). The converse implication is proved in the following
\begin{theorem}
$S$ is recurrent $\implies$ $S$ is reversal invariant.
\label{rever}
\end{theorem}
\begin{proof}
Suppose that $S$ is recurrent. Then, by Lemma \ref{recur}, for every integer $k$ the sequence $S$ has a $k$-regular prefix $w_k$. We have
\begin{equation}
S=w_k^{-2}12\dots
\label{revers}
\end{equation}
where it is easily seen that $w_k^{-2}$ must end with 2. Defining $v$ by $v2=w_k^{-2}$ we can write
\begin{equation}
S=v212\dots
\label{reverz}
\end{equation}
Integrating \eqref{reverz}, and recalling that $v2\in (k-2)$-R, we have
\begin{equation}
S=(v2)^{-1}(12)^{-1}\dots
\label{revenj}
\end{equation}
and as $|v2|$ and $|v212|$ are both even, the 2 appearing as the last element of $v2$ is transformed by the rules \eqref{subs} in the same way as the 2 appearing as the last element of $v212$. Therefore prefix $(v2)^{-1}1$ can be rewritten as
\begin{equation*}
v^{-1}\overleftarrow{(12)^{-1}}
\end{equation*}
Proceeding by induction, suppose that the prefix $(v2)^{-h}1$ can be rewritten as 
\begin{equation}
v^{-h}\overleftarrow{(12)^{-h}}
\label{ll}
\end{equation}
Integrating $h$ times \eqref{reverz}, and recalling that $v2$ is $(k-2)$-regular, we get the prefix 
\begin{equation}
(v2)^{-h}(12)^{-h}
\label{lll}
\end{equation}
Set $n:=|(v2)^{-h}|$. Comparing \eqref{ll} and \eqref{lll}, it follows that, for every integer $j$ such that $0\le j\le |(12)^{-h}|$, the $(n-j)$th element of $(v2)^{-h}$ is equal to and has always the same parity as the $(j+2)$th element of $(12)^{-h}$, and therefore is transformed by the rules \eqref{subs} in the same way. Therefore, since integrating \eqref{lll} we get $(v2)^{-h-1}(12)^{-h-1}$, integrating \eqref{ll} we must get the prefix
\begin{equation}
v^{-h-1}\overleftarrow{(12)^{-h-1}}
\label{llll}
\end{equation}

Since $(12)^{-k}$ is a prefix of $S$ for every $k$ by Lemma \ref{lemma1}, the arbitrariness of $k$, and thus of $h$, allows us to conclude that $S$ is reversal invariant.
\end{proof} 

A natural question is whether, for a subrow $(w,n)$, the property of being $k$-regular for large values of $k$ is compatible with the requirement $w\in C^\infty$. In fact it is possible to prove more, i.e., that $S$ can be eventually written as a concatenation of arbitrarily regular subrows. More precisely, we have the following
\begin{theorem}
For every non-negative integer $n$, there exist a finite word $u_n$ and finite words $w_i$ ($i=1,2\dots$) such that 
\begin{equation}
S=u_nw_1w_2\dots
\label{covering}
\end{equation}
where the subrows $w_i\in k$-R for every $i$ and $k\ge n$. 
\label{cover}
\end{theorem}
\begin{proof}
We proceed by induction. Let us suppose that $S$ has the form \eqref{covering} and that $w_i\in k$-R $\forall i$. Let $M$ be the set of positive integers $i_m$ such that $w_{i_m}\notin (k+1)$-R. Clearly the subrows $w_{i_m}$ are $k$-normal. If $M$ is finite, we define $p:=\max\left\lbrace j\in\mathbb{N}^+:j \in M\right\rbrace$ and $u_{n+1}:=u_nw_1\dots w_p$ so that
$$S=u_{n+1}w_{p+1}w_{p+2}\dots\quad (p=1,2\dots)$$
which is the desired result. 

If $M$ is infinite, $i_m$ is a subsequence of $i$, so for every positive integer $m$ we can define the words $v_m:=w_{i_m}w_{(i_m +1)}\dots w_{i_{(m+1)}}$. Every $v_m$ is a concatenation of the $k$-normal subrow $w_{i_m}$, the (possibly empty) word formed by the $i_{(m+1)}-i_m-1$ words $w_{i_m+h}$ ($h=i_m+1\dots i_{m+1}-1$), which are $(k+1)$-regular, and the $k$-normal subrow $w_{i_{(m+1)}}$.
Therefore, by Lemma \ref{Syntax}, $v_m\in(k+1)$-R for every $m$, and therefore defining $p:=\min M$ and the word $u_{n+1}:=u_nw_1w_2\dots w_{p-1}$, we have
\begin{equation}S=u_{n+1} v_1 v_3 \dots v_{2m+1}\dots\quad(m=1,2\dots)
\label{vconm}
\end{equation}
which is the desired result.

Finally, $S$ is obviously written as a concatenation of 0-regular subrows, so the proof is concluded.
\end{proof}
\begin{rem} In the previous Theorem, if we start the inductive construction of the words $w_i$ from $u_0:=\epsilon$ and $w_i:=s_{2i-1}s_{2i}$ ($i=1,2\dots$), we can have two different possibilities:
\begin{enumerate}
\item $u_n=\epsilon$ for every non-negative integer $n$. In this case $S$ is written as a concatenation of $k$-regular subrows for every $k$, and therefore it is recurrent and reversal invariant by Lemmas \ref{recur} and \ref{mirroring}.
\item At some step $\bar{n}$ of the inductive procedure we have $\epsilon\ne u_{\bar{n}}\in (\bar{n}-1)$-N. By Lemma \ref{Syntax} it follows that in this case, continuing the inductive procedure, we have $u_n\in (\bar{n}-1)$-N for every $n>\bar{n}$.
\end{enumerate}
\label{rem3}
\end{rem}
\begin{rem}
We can apply the iterative procedure of Theorem \ref{cover} starting from an arbitrary element of $S$, as no special properties of the beginning of $S$ were used. This means that, for every positive integer $k$, the same conclusion of the Lemma applies to the sequence $s_k s_{k+1}s_{k+2}\dots$ .
\label{rempre}
\end{rem}

To proceed further we need one more definition, as we want to assign a special name to the subrows $v_i$ in \eqref{blocks}. We recall that, for every prefix $w$ of $S$ we have, by Lemma \ref{furtherintegration}, that $w^{-k}$ is also a prefix. 
\begin{defi} For every prefix $w$ such that $|w|>1$, and for every positive integer $k$, we define the $k$-th block generated by the prefix $w$ as the unique subrow  $b_k$ such that $w^{-k+1}b_k=w^{-k}$. \label{blocksdefi}\end{defi}
 We have clearly 
\begin{equation}
S=wb_1b_2b_3\dots
\label{blocksdef}
\end{equation}
and 
\begin{equation}(wb_1\dots b_{k-1})^{-1}=w^{-k}=wb_1\dots b_{k}\label{blocksdef2}\end{equation}
 Notice that, for every $k$, $b_{k+1}=\left(b_k\right)^{-1}_S$ so that, for every $k$: $$b_k=(b_{k-1})^{-1} \quad \text{if}\quad |wb_1\dots b_{k-2}|\quad\text{is even}$$
$$b_k=\widetilde{(b_{k-1})^{-1}} \quad \text{if}\quad |wb_1\dots b_{k-2}|\quad\text{is odd}$$
Since $|w|=|\widetilde{w}|$ for every finite word $w$, we have by Lemma \ref{lengthestimates}, that \begin{equation}\frac{6}{5}|b_{k-1}|\le|b_k|\le \frac{9}{5}|b_{k-1}|\label{bbbbbb}\end{equation}
and therefore
\begin{equation}
\left(\frac{6}{5}\right)^{k-1}|b_1|\le|b_k|\le\left(\frac{9}{5}\right)^{k-1}|b_1|\label{blockslength}
\end{equation}
 A block has the property of being always ``not too small" with respect to whatever comes before it in the sequence, and therefore the asymptotic frequencies of 1s and 2s, if they exist, have to be reached uniformly on the blocks. More precisely, we have the following

\begin{lemma}
Let $w$ be a prefix ($|w|\ge 2$) of $S$ and $b_k$ ($k=1,2\dots$) the blocks generated by $w$. Suppose $f_{\infty}(1)$ exists. Then for every $\epsilon>0$ there is an integer $n$ such that $|f_{b_k}(1)-f_{\infty}(1)|<\epsilon$ for every $k\ge n$. 
\label{uniff}
\end{lemma}
\begin{proof}
For every positive integer $k$, let us define the prefixes $u_k:=wb_1\dots b_k$. Using \eqref{blockslength}, it can be shown that there exist two real numbers $c_1$ and $c_2$, with $0<c_1<c_2<1$ such that, for every $k$ large enough,
\begin{equation}|b_k| \ge c_1 |u_k|\qquad\text{and}\qquad
|b_k| \le c_2 |u_k|
\label{c_1c_2}
\end{equation}
Indeed, setting $\rho_k:= \frac{|b_{k+1}|}{|b_{k}|}$, \eqref{bbbbbb} implies that the ratio $\frac{|u_k|}{|b_k|}$ is bounded from above as follows: 

\begin{equation*}
\frac{|u_k|}{|b_k|}\le\frac{|w|+|b_1|\left(1+\sum_{j=1}^{k-1}\prod_{i=1}^{j}\rho_i\right)}{|b_1|\prod_{i=1}^{k-1}\rho_i}
\end{equation*}
so that, noticing that  by Lemma \ref{lengthestimates} we have $|w^{-1}|=|w|+|b_1|\le 6|b_1|$, it follows
\begin{equation*}
\frac{|u_k|}{|b_k|}\le 1+ \frac{6}{\prod_{i=1}^{k-1}\rho_i}+\sum_{j=2}^{k-1}\left(\prod_{i=j}^{k-1}\rho_i\right)^{-1}
\end{equation*}
Since $\rho_i\ge \frac{6}{5}$ for every positive integer $i$, it follows 
\begin{equation*}
\frac{|u_k|}{|b_k|}\le 1+6\left( \frac{5}{6} \right)^{k-1}+\sum_{i=1}^{k-2}\left(\frac{5}{6} \right)^{i} \xrightarrow[k \to \infty]{} 6
\end{equation*}
Therefore, $c_1$ is bounded from below by $\frac{1}{6}$ and in particular is bounded away from zero (it is proved similarly that $c_2$ is bounded away from 1).

We have:
\begin{equation}
f_{u_k}(1)=f_{u_{k-1}}(1) \frac{|u_{k-1}|}{|u_k|}+f_{b_k}(1)\frac{|b_k|}{|u_k|}\ge f_{u_{k-1}}(1) (1-c_2)+f_{b_k}(1)c_1
\label{keq}
\end{equation}
Since $u_k$ is a prefix for every $k$ and by \eqref{integrallength} $|u_k|\rightarrow \infty$ when $k\rightarrow\infty$, if there exists $f_{\infty}(1)$ then for every $\epsilon>0$ there is $h$ so large that, for every $k>h$,
$$f_{u_{k-1}}(1)=f_{\infty}(1)+{\epsilon}_1$$
and
$$f_{u_{k}}(1)=f_{\infty}(1)+{\epsilon}_2$$
with $\max\left\lbrace\epsilon_1,\epsilon_2\right\rbrace<\epsilon$. Therefore, from \eqref{keq} we have
\begin{equation*}
f_{b_k}(1)c_1\le f_{u_k}(1)-f_{u_{k-1}}(1) (1-c_2)=\epsilon_2-\epsilon_1+f_{\infty}(1)c_2+\epsilon_1 c_2
\end{equation*}
whence
\begin{equation*}
f_{b_k}(1)c_1-f_{\infty}(1)c_2\le \epsilon_2+\epsilon_1(1-c_2)
\end{equation*}
so that
\begin{equation}
\left( f_{b_k}(1)-f_{\infty}(1) \right)c_1\le \epsilon_2+\epsilon_1(1-c_2)+f_{\infty}(1)(c_2-c_1)
\label{final}
\end{equation}
If there exists $f_{\infty}(1)$ the difference $c_2-c_1$ becomes arbitrarily small when $k$ diverges. Indeed:
\begin{equation}
\frac{|b_k|}{|u_{k-1}|}=f_{u_{k-1}}(1)+2\left(1-f_{u_{k-1}}(1)\right)
\label{xxxx}
\end{equation}
The right hand side of \eqref{xxxx} tends to $2-f_{\infty}(1)=1+f_{\infty}(2)$ when $k\rightarrow \infty$. Therefore also $\frac{|b_k|}{|u_k|}$ converges to a limit when $k$ diverges, as it is obviously $|u_k|=|u_{k-1}|+|b_k|$. Therefore, if $k$ is large enough, we can take $c_1$ and $c_2$ such that $c_2-c_1$ is arbitrarily small, and since $c_1$ is bounded away from 0, \eqref{final} implies that $|f_{b_k}(1)-f_{\infty}(1)|$ is vanishingly small when $k$ diverges.
\end{proof}
\begin{rem}
Let us take another prefix $v$ with $|v|>|w|$ and let $d_k$ denote the blocks generated by $v$. Then instead of \eqref{xxxx} we have
\begin{equation}
\frac{|d_k|}{|p_{k-1}|}=f_{p_{k-1}}(1)+2\left(1-f_{p_{k-1}}(1)\right)
\label{xxxxx}
\end{equation}
where $p_k:=vd_1\dots d_k$. Clearly the right hand side of \eqref{xxxxx} converges to $1+f_{\infty}(2)$ faster than the right hand side of \eqref{xxxx}, as for every $k$ we have $|p_k|>|u_k|$. From this it easily follows that for every $\epsilon>0$, if $n$ satisfies Lemma \ref{uniff} for a given prefix $w$, then it satisfies Lemma \ref{uniff} for $v$, i.e. $|f_{b_k}(1)-f_{\infty}(1)|<\epsilon$ for every $k\ge n$ implies $|f_{d_k}(1)-f_{\infty}(1)|<\epsilon$ for every $k\ge n$. 
\label{remblockpref}  
\end{rem}

\begin{defi}
We say that a prefix is $k$-\emph{minimal} if it is the shortest $k$-normal prefix of $S$.
\end{defi}
 Clearly for every positive integer $k$ there is at most one $k$-minimal prefix.

We want now to provide sufficient conditions (as weak as possible) implying that Keane's conjecture is true. Specifically, we require the existence of arbitrarily normal prefixes as well as a ``relaxed uniformity" property, i.e. that a sufficiently large portion of every block belonging to a certain subset (the ones generated by $k$-minimal prefixes) is representative of the frequency of 1s on that block. 

More precisely, we have the following

\begin{theorem}
Suppose that
\begin{enumerate}
\item There exists $f_{\infty}(1)$; 
\item There is a strictly increasing sequence of positive integers $$K:=k_1,k_2,k_3\dots$$ such that there exists a $k_n$-normal prefix $p_{k_n}$ of $S$ for every $n$;
\item for every $\epsilon>0$, there is a positive integer $L_\epsilon$ such that $|f_{b_{m}^n}(1)-f_{c_{m}^n}(1)|<\epsilon$ 

for every positive integer $m$ such that $|b_m^n|>L_\epsilon$, where
\begin{itemize}
\item $b_{m}^n$ ($m=1,2,\dots$) are the blocks generated by the $k_n$-minimal prefixes $p_{k_n}$,
\item $c_{m}^n$ is a prefix of the block $b_{m}^n$ such that $|c_{m}^n|\ge L_\epsilon$. 
\end{itemize}

\vspace{0.2cm}

Then $f_{\infty}(1)=\frac{1}{2}$.
\end{enumerate}
%\begin{rem}
%Notice that the portion $|c_{k_n}|$ of the $k$-th block required to be representative of the density of 1s in that block is allowed to become arbitrarily large with $k_n$.  
%\end{rem}

\end{theorem}
\begin{proof}

Take $\epsilon>0$ and consider a prefix $w$ so long that $|w|\ge L_\epsilon$ and \begin{equation}
|f_v(1)-f_{\infty}(1)|<\epsilon
\label{preff}
\end{equation} 
for every prefix $v$ such that $|v|\ge|w|$. 
We can find a $k_n$-minimal prefix $p$ with $k_n$ being the smallest element of the sequence $K$ such that $w$ is a prefix of $(12)^{-k_n+2}$, which implies that $p^{-k_n}w$ is a prefix of $S$. By Lemma \ref{uniff} there is a positive integer $m$ such that the $m$-th element $k_m$ of the sequence $K$ has the property that $|f_{b_j}(1)-f_{\infty}(1)|<\epsilon$ for every $j\ge k_m$, where $b_j$ are the blocks generated by $p$. There are two possibilities:
\begin{enumerate}
\item $k_m>k_n$.

 Then by hypothesis 2. we can find a $k_s$-minimal prefix $q$ where $k_s$ is the smallest element of the sequence $K$ such that $k_s\ge k_m$ and $|q|>|p|$ (of course the latter is verified for some $k_s$ because there are only finitely many prefixes which are shorter than $p$). Clearly $q^{-k_s}w$
is also a prefix of $S$. Moreover, recalling Remark \ref{remblockpref} and denoting by $d_j$ the blocks generated by $q$, we have 
\begin{equation}|f_{d_j}(1)-f_{\infty}(1)|<\epsilon 
\label{d_k}
\end{equation}
for every $j\ge k_m$ and thus in particular for $j\ge k_s$. We can assume that $k_s\ge 2$, so that $|q|\ge 16$, $|d_1|\ge 8$ and thus 
\begin{equation}
q^{-k_s}d_{k_s+1}=q^{-k_s}wu
\label{eqx}
\end{equation}
with $u$ nonempty.

Since $|q^{-k_s-1}|$ is odd by hypothesis, integrating two more times \eqref{eqx} we get the prefix $$q^{-k_s-2}d_{k_s+3}=q^{-k_s-2}\widetilde{w^{-2}}{u_S^{-2}}$$ where $u_S^{-2}$ is nonempty, so that $\widetilde{w^{-2}}$ is a prefix of $d_{k_s+3}$ (and does not coincide with it).
%If $\bar{k}$ is the least integer such that $w$ is a prefix of $(12)^{-\bar{k}}$, setting $k_s-\bar{k}=:\delta$, we have that $\widetilde{w^{-\delta-2}}$ is a prefix of $d_{{k_s+3}}$ (notice that it can be $\delta=0$). 
Since $|w^{-2}|>|w|\ge L_\epsilon$, by assumption the following inequalities are verified:
\begin{equation}
|f_{d_{k_s+3}}(1)-f_{\infty}(1)|<\epsilon
\label{comparings}
\end{equation}
\begin{equation}
|f_{w^{-2}}(1)-f_{\infty}(1)|<\epsilon
\label{comparing}
\end{equation}
%We can bound $|d_{k_s+3}|$ from above by means of \eqref{blockslength} and Lemma \eqref{lengthestimates} whence, recalling that $|q^{-1}|=|q|+|d_{1}|$, we obtain
%\begin{equation}
%|d_{k_s+3}|\le \left(\frac{9}{5}\right)^{k_s+2}|d_1|\le \left(\frac{9}{5}\right)^{k_s+2} \frac{4}{5}|q| 
%\end{equation}
%Moreover, we can bound from below $|\widetilde{w^{-\delta-2}}|$. Indeed we have $|w|>|(12)^{-\bar{k}+1}|$, thus $$|w^{-\delta-2}|>|(12)^{-k_s-1}|\ge 2\left(\frac{6}{5}\right)^{k_s+1}$$
%where the last inequality holds because of Lemma \ref{lengthestimates}.
%
% Therefore the ratio $\frac{|d_{k_s+3}|}{|w^{-\delta-2}|}$ is bounded from above by $\left(\frac{3}{2}\right)^{k_n+2}|q|$, and therefore,
 
and moreover, by hypothesis 3., $$|f_{\widetilde{w^{-2}}}(1)-f_{d_{k_n+3}}(1)|<\epsilon$$ Combining the last inequality with \eqref{comparings} and \eqref{comparing}, we get that the difference $|f_{\widetilde{w^{-2}}}(1)-f_{w^{-2}}(1)|$ is vanishingly small if $n$ is large enough, and therefore so is, by definition of mirror word, the difference $|f_{w^{-2}}(1)-f_{w^{-2}}(2)|$, from which    we can conclude.

\item $k_n\ge k_m$.

Then we proceed as above with the prefix $p$ instead of $q$ and $k_n$ instead of $k_m$.
\end{enumerate}
\end{proof}
\section{An iterative procedure providing arbitrarily long recurrent subwords}
\label{ite}
In this section we want to use the iterative construction shown in the proof of Theorem \ref{cover} to establish constructively the existence of recurrent subwords of arbitrary length and identify places where they must appear in the structure of $S$. Before this, let us expicitly recall that, for any aperiodic sequence and every positive integer $n$, the existence of at least $n+1$ distinct subwords of length $n$ (which are easily shown to be recurrent) is a basic combinatorial result (see for instance \cite{lothaire2002algebraic}).
%\begin{itemize}
%\item the existence of at least one recurrent subword of any given finite length is shown by a trivial cardinality argument;
%\item the existence of two recurrent subwords of any given length is easily proven using the fact that $S$ cannot be eventually periodic;
%\item the existence of more than two recurrent subwords of any given length is also easy. Indeed, assuming that only two recurrent subwords $w_1$ and $w_2$ of length $L$ exist, one can deduce the existence of other recurrent subwords which are subwords of $w_1w_2$ or of $w_2w_1$, thus obtaining a contradiction.   
%\end{itemize}
However, being so general, this kind of argument is of course non-constructive, not providing any insight about where to find such recurrent subwords in the sequence, nor about the structure of such subwords themselves.

In order to obtain a bit more, let us start by associating to every subrow of $S$ a sequence over $\left\lbrace 0,1\right\rbrace^{\omega}$ describing the parity of all its $S$-integrals. More precisely, we introduce the following definition:

\begin{defi}
For every subrow $w$ we define $P_0(w)=0$ if $|w|$ is even, $P_0(w)=1$ otherwise. We define inductively $P_n(w)=0$ if $|w^{-n+1}_S|$ is even, $P_n(w)=1$ otherwise. We call the sequence $P_n(w)$ the \emph{history of parity} of the integrals of $w$.
\end{defi}
 In the particular case in which $w$ is a prefix, $P_n(w)$ is simply the sequence describing the parities of $|w^{-n}|$. 

\begin{lemma}
Suppose that $u_1$ and $u_2$ are distinct prefixes of $S$ such that $S=u_1w_1\dots$ and $S=u_2w_2\dots$ with the subrows $w_1$ and $w_2$ coinciding as subwords. If there exists $\bar{n}$ such that $P_{\bar{n}}(u_1)\ne P_{\bar{n}}(u_2)$, then $(w_1)_S^{-k}\ne (w_2)_S^{-k}$ for every $k\ge \bar{n}$.  
\label{noname}
\end{lemma}
\begin{proof}
Suppose in particular that $\bar{n}$ is the least integer for which $P_n(u_1)\ne P_n(u_2)$. Then it follows that $(w_1)^{-\bar{n}+1}_S=(w_2)^{-\bar{n}+1}$, while, recalling the substitution rules \eqref{subs}, we have $(w_1)^{-\bar{n}}_S= \widetilde{(w_2)^{-\bar{n}}_S}\ne (w_2)^{-\bar{n}}_S$. To conclude it is enough to observe that if $w$ and $v$ are nonempty subwords, $w\ne v$ implies that the four words $w^{-1}$, $\widetilde{w^{-1}}$, $v^{-1}$ and $\widetilde{v^{-1}}$ are all distinct.
\end{proof}
 Let us now start the inductive procedure described in the proof of Theorem \ref{cover} with the empty prefix $u_0=\epsilon$ and $w_i:=s_{2i-1}s_{2i}$ ($i=1,2\dots$). Suppose that $\bar{n}$ is the least integer for which $u_{\bar{n}}\ne \epsilon$ (we recall that, by Lemma \ref{recur}, if $u_n=\epsilon$ for every positive integer $n$, then every subword of $S$ is recurrent). It follows from the construction of Theorem \ref{cover} that it has to be $u_{\bar{n}+k}\in\bar{n}\text{-}N$ for every positive integer $k$. By direct inspection it can be seen that $\bar{n}$ is not smaller than 2, as the prefix $s_1s_2\dots s_{16} = 12 21 12 12 21 22 11 21$ is 2-regular. Therefore, according to Theorem \ref{cover}, we have that, for every positive integer $k$, 
\begin{equation}
S=u_k w_1w_2\dots
\label{111}
\end{equation} 
 with $u_k\in h$-N ($h\ge 2$) and $w_i\in k$-R for every $i$. Since $u_k$ is 1-regular, we have that, writing $S$ as
\begin{equation}
S=u_{k}^{-2} (w_1)_S^{-2} (w_2)_S^{-2}\dots 
\label{eqqqq}
\end{equation}
the subrows $(w_i)_S^{-2}$ all begin with 12. Therefore integrating $k-2$ times \eqref{eqqqq} and suitably defining the subrows $v_i$, we obtain for $S$ the structure
\begin{equation*}
S=v_0(z_1)_S^{-k+2}v_1(z_2)_S^{-k+2}\dots
\end{equation*}
were $z_i=12$ for every integer $i$, and the subrows $(z_i)_S^{-k+2}$ are all coinciding as subwords since, for every $j\le k$, the parity of $|(u_kw_1w_2\dots w_n)^{-j}|$ is the same for every $n>0$. Finally, since $k$ is arbitrarily large by Theorem \ref{cover}, \eqref{integrallength} implies that the recurrent subwords $(z_i)_S^{-k+2}$ are arbitrarily long.

We can also use the same iterative procedure to identify other arbitrarily long recurrent subwords which, in general, are not coinciding with the previous ones. Indeed, recalling Remark \ref{rempre}, we can also apply the iterative construction starting right after any given prefix of $S$. Taking, for instance, the 1-normal prefix $p:=$1221, for every positive integer $k$ we can write
\begin{equation*}
S=p\bar{u}_k \bar{w}_1 \bar{w}_2\dots  
\end{equation*}
 where, noticing that $s_5\dots s_8\in 2$-R, we have $\bar{u}_k\in \bar{h}$-N ($\bar{h}\ge 2$) and $\bar{w}_i\in k$-R for every $i$. Since $p\bar{u}_k$ is 1-regular by Lemma \ref{Syntax}, we have that, writing $S$ as
\begin{equation}
S=(p\bar{u}_{k})^{-2} (\bar{w}_1)_S^{-2} (\bar{w}_2)_S^{-2}\dots 
\label{eqqq2}
\end{equation}
the subrows $(\bar{w}_i)_S^{-2}$ all begin with 12. Therefore, integrating \eqref{eqqq2} $k-2$ times and suitably defining the subrows $\bar{v}_i$, we obtain for $S$ the structure
\begin{equation*}
S=\bar{v}_0(\bar{z}_1)_S^{-k+2}\bar{v}_1 (\bar{z}_2)_S^{-k+2}\dots
\end{equation*} 
were $\bar{z}_i=12$ for every integer $i$, and the subrows $(\bar{z}_i)_S^{-k+2}$ are all coinciding as subwords since, for every $j\le k$, the parity of $|(p\bar{u}_k\bar{w}_1\bar{w}_2\dots \bar{w}_n)^{-j}|$ is the same for every $n>0$.
Since by construction $P_n(u_k)\ne P_n(p\bar{u}_k)$, Lemma \ref{noname} ensures that $(z_i)_S^{-k+2}\ne (\bar{z}_i)_S^{-k+2}$, while again since $k$ can be arbitrarily large (from Theorem \ref{cover}), the recurrent subwords $(\bar{z}_i)_S^{-k+2}$ have arbitrarily large length by \eqref{integrallength}.

\section{More open questions}
\label{more}
 How do $k$-regular prefixes (or, in general, subrows) look? This is a difficult question. Let us take a look at the first cases. A prefix $w=s_1\dots s_n$ is
\begin{itemize}
\item 0-regular if $n$ is even;
\item 1-regular if it is 0-regular and $\Sigma w$ is even;
\item 2-regular if it is 1-regular and $\Sigma_{i=0}^{\frac{n}{2}-1} s_{2i+1}$ is even;
\item 3-regular if it is 2-regular and 
\begin{align*}
&\Big|\left\lbrace s_j: s_j=2\,\text{and}\,\, j\,\, \text{is odd} \right\rbrace\Big| \quad \qquad+ \\
&\Big|\left\lbrace s_j : s_j=1,\,\, j\,\, \text{is odd} \,\,\text{and}\,\, \Sigma_{i=1}^{j-1} s_i \,\,\text{is even}\right\rbrace\Big|
\end{align*}
is even.
\end{itemize}
With some effort one can go a bit further, but it is not easy to see where the thing is going.

From numerical computations one gets the impression that the requirement of being $k$-regular for large $k$ is quite hard to meet - and we recall that recurrence for $S$ is equivalent to the existence of \emph{arbitrarily} regular prefixes. For instance, the shortest 10-regular prefix has length 6410, while the 10-minimal prefix has length 7144. Since the number of independent conditions that a finite word has to satisfy to be $k$-regular seems to increase with $k$, the following conjecture arises naturally.
\begin{conj}
There are no $\infty-$regular subrows in $\mathcal{SR}(S)$.
\label{conj1}
\end{conj}
A consequence of this conjecture is seen in Theorem \ref{thm1}. It is also natural, in our view, to formulate a stronger conjecture, namely that two subrows whose integrals have exactly the same history of parity, must coincide. More precisely, we state the following
\begin{conj}
If $u$ and $w$ are two subrows and $P_n(u)=P_n(w)$ for every $n\ge 0$, then $u=w$.   
\label{conj2}
\end{conj}
If this is true, then Conjecture \ref{conj1} follows, as if $w$ is an $\infty$-regular subrow, we can split it in two subrows with the same history of parity, which contradicts Conjecture \ref{conj2}.

\section{Acknowledgments}
 I am deeply grateful to Lucio Russo (who introduced me to the problem), Stefano Isola and Riccardo Piergallini for many fruitful discussions.

\bibliographystyle{plain}
\bibliography{biblio}
\end{document}